\documentclass[12pt,oneside,english]{amsart}
\usepackage[latin1]{inputenc}
\usepackage{amsthm}
\usepackage{amssymb}
\usepackage{graphicx}

\makeatletter

\providecommand{\tabularnewline}{\\}

\numberwithin{equation}{section}
\numberwithin{figure}{section}
\theoremstyle{plain}
\newtheorem{thm}{\protect\theoremname}
  \theoremstyle{plain}
  \newtheorem{lem}[thm]{\protect\lemmaname}
  \theoremstyle{remark}
  \newtheorem{rem}[thm]{\protect\remarkname}
  \theoremstyle{plain}
  \newtheorem{prop}[thm]{\protect\propositionname}

\usepackage{babel}

\makeatother

  \providecommand{\lemmaname}{\inputencoding{latin9}Lemma}
  \providecommand{\propositionname}{\inputencoding{latin9}Proposition}
  \providecommand{\remarkname}{\inputencoding{latin9}Remark}
\providecommand{\theoremname}{\inputencoding{latin9}Theorem}

\begin{document}

\title{ODE Solvers Using Band-limited Approximations}

\author{G. Beylkin$^{*}$ and K. Sandberg$^{**}$ }

\address{$^{*}$Department of Applied Mathematics \\
 University of Colorado at Boulder \\
 526 UCB \\
 Boulder, CO 80309-0526 \\
$^{**}$Computational Solutions, Inc.\\
1800 30th Street, Suite 210B\\
Boulder, Colorado 80301}
\begin{abstract}
We use generalized Gaussian quadratures for exponentials to develop
a new ODE solver. Nodes and weights of these quadratures are computed
for a given bandlimit $c$ and user selected accuracy $\epsilon$,
so that they integrate functions $e^{ibx}$, for all $|b|\le c$,
with accuracy $\epsilon$. Nodes of these quadratures do not concentrate
excessively near the end points of an interval as those of the standard,
polynomial-based Gaussian quadratures. Due to this property, the usual
implicit Runge-Kutta (IRK) collocation method may be used with a large
number of nodes, as long as the method chosen for solving the nonlinear
system of equations converges. We show that the resulting ODE solver
is symplectic and demonstrate (numerically) that it is A-stable. We
use this solver, dubbed Band-limited Collocation (BLC-IRK), for orbit
computations in astrodynamics. Since BLC-IRK minimizes the number
of nodes needed to obtain the solution, in this problem we achieve
speed close to that of the traditional explicit multistep methods. 
\end{abstract}

\thanks{This research was partially supported by AFOSR grant FA9550-07-1-0135,
NSF grant DMS-0612358, DOE/ORNL grants 4000038129 and DE-FG02-03ER25583. }


\maketitle

\section{Introduction}

Current methods for solving ODEs, be that multistep or Runge-Kutta,
are based on polynomial approximations of functions. However, both
classical and recent results \cite{SLE-POL:1961,LAN-POL:1961,SLEPIA:1983,XI-RO-YA:2001,BEY-MON:2002,BEY-SAN:2005,BO-GA-SA:2013,OS-RO-XI:2013}
indicate that in many situations band-limited functions provide a
better tool for numerical integration and interpolation of functions
than the traditional polynomials. We construct a new method for solving
the initial value problem for Ordinary Differential Equations (ODEs)
using band-limited approximations and demonstrate certain advantages
of such approach. As an example, we consider orbit computations in
astrodynamics as a practical application for the new ODE solver as
well as a gauge to ascertain its performance.

It is well-known that choosing between equally spaced and unequally
spaced nodes on a specified time interval, in other words, choosing
a multistep vs a collocation based Runge-Kutta method, results in
significantly different properties of ODE solvers. For example, multistep
schemes have $\left\{ \mathcal{R}e(z)\le0,\,\, z\in\mathbb{C}\right\} $
as the region of absolute stability (A-stable) only if their order
does not exceed $2$, the so-called Dahlquist barrier. In contrast,
an A-stable implicit Runge-Kutta (IRK) scheme may be of arbitrary
order. A class of A-stable IRK schemes uses the Gauss-Legendre quadrature
nodes on each time interval and the order of such methods is $2\nu$,
where $\nu$ is the number of nodes (see, e.g., \cite{ISERLE:1996}).
A-stability assures that growth and decay of numerical solutions exactly
mimics that of the analytic solutions of the test problem which, in
turn, implies that the choice of step size involves only accuracy
consideration. 

Another numerical property of interest, that of preservation of volume
in the phase space, identifies symplectic integrators. Symplectic
integrators preserve a particular conserved quantity of Hamiltonian
systems as well as an approximate Hamiltonian. In problems of orbit
determination, a symplectic integrator would maintain the correct
orbit more or less indefinitely with the error accumulating only in
a position along that orbit, thus closely reproducing a particular
behavior of analytic solutions of nonlinear Hamiltonian systems. We
note that IRK schemes which use the Gauss-Legendre nodes are symplectic
(see, e.g., \cite{ISERLE:1996}).

While IRK schemes with the Gauss-Legendre nodes provide an excellent
discretization of a system of ODEs, using many such nodes on a specified
time interval is not practical. The nodes of the Gauss-Legendre quadratures
(as well as any other polynomial based Gaussian quadratures) accumulate
rapidly towards the end points of an interval. A heuristic reason
for such accumulation is that these quadratures have to account for
a possible rapid growth of polynomials near the boundary. 

In this paper we demonstrate that, within IRK collocation methods,
quadratures based on polynomials may be replaced by quadratures for
band-limited exponentials. The nodes of these quadratures do not accumulate
significantly toward the end points of an interval (a heuristic reason
for an improved arrangement of nodes is that the exponentials do not
grow anywhere within the interval). Our method addresses numerical
integration of ODEs whose solutions are well approximated by band-limited
exponentials. Band-limited exponentials have been successfully used
in problems of wave propagation \cite{BEY-SAN:2005} (see also \cite{SAN-WOJ:2011}),
where it is natural to approximate solutions by band-limited functions.
While solutions of ODEs are typically well approximated by band-limited
exponentials as well, there may be exceptions since some ODEs may
have polynomial solutions. In such cases the use of polynomial based
quadratures may be more efficient.

Unlike the classical Gaussian quadratures for polynomials that integrate
exactly a subspace of polynomials up to a fixed degree, the Gaussian
type quadratures for exponentials use a finite set of nodes to integrate
an infinite set of functions, namely, $\left\{ e^{ibx}\right\} _{\left|b\right|\le c}$
on the interval $\left|x\right|\le1$. As there is no way to accomplish
this exactly, these quadratures are constructed so that all exponentials
with $|b|\le c$ are integrated with accuracy of at least $\epsilon$,
where $\epsilon$ is arbitrarily small but finite. Such quadratures
were constructed in \cite{BEY-MON:2002} and, via a different approach
in \cite{XI-RO-YA:2001} (see also \cite{RE-BE-MO:2013}). As observed
in \cite{BEY-SAN:2005}, quadrature nodes of this type do not concentrate
excessively toward the end of the interval. The density of nodes increases
toward the end points of the interval only by a factor that depends
on the desired accuracy $\epsilon$ but not on the overall number
of nodes. 

Using quadratures to integrate band-limited exponentials with bandlimit
$2c$ and accuracy $\epsilon^{2}$, we naturally arrive at a method
for interpolation of functions with bandlimit $c$ and accuracy $\epsilon$
(see \cite{XI-RO-YA:2001,BEY-MON:2002}). It turns out that the nodes
and weights of quadratures to interpolate with accuracy $\epsilon\approx10^{-15}$
can, in fact, be constructed using only the standard double precision
machine arithmetic. However, generating the integration matrix for
the new double precision BLC-IRK method requires using quadruple precision
in the intermediate calculations. Importantly, once generated, the
quadratures and the integration matrix are applied using only the
standard double precision.

While analytically the classical Gauss-Legendre quadratures for polynomials
are exact, in practice their accuracy is limited by the machine precision.
By choosing (interpolation) accuracy $\epsilon\approx10^{-15}$, our
integrator is effectively ``exact'' within the double precision
of machine arithmetic. Remarkably, using a particular construction
of the integration matrix, we show that BLC-IRK method is (exactly)
symplectic and, with high accuracy, A-stable. This result was unexpected
and indicates that properties of approximate quadratures for band-limited
exponentials need to be explored further.

While IRK schemes require solving a system of nonlinear equations
at each time step, it does not automatically imply that such schemes
are always computationally more expensive than explicit schemes. In
the environment where the cost of function evaluation is high, the
balance between the necessary iteration with fewer nodes of an implicit
scheme vs significantly greater number of nodes of an explicit scheme
(but no iteration), may tilt towards an implicit scheme. In problems
of astrodynamics, we use an additional observation that most iterations
can be performed with an inexpensive (low fidelity) gravity model,
making implicit schemes with a large number of nodes per time interval
practical. We select the problem of orbit determination as an example
where our approach is competitive with numerical schemes that are
currently in use (see \cite{B-J-B-A:2012p,B-J-B-S-A:2013}). We take
advantage of the reduced number of function calls to the full gravity
model in a way that appears difficult to replicate using alternative
schemes.

In order to accelerate solving a system of nonlinear equations, we
modify the scheme by explicitly exponentiating the linear part of
the force term. For the problem of orbit computations this modification
accelerates convergence of iterations by (effectively) makes use of
the fact that the system is of the second order. So far we did not
study possible acceleration of iterations using spectral deferred
correction approach as in \cite{DU-GR-RO:2000,LAY-MIN:2005,HU-JI-MI:2006,JIA-HUA:2008}. 

We start by providing background information on quadratures for band-limited
functions in Section~\ref{sec:Preliminaries:-quadratures-for}. We
then describe BLC-IRK method in Section~\ref{sec:BQC-IRK-method}
(with some details deferred to Appendix). In Section~\ref{sec:Applications}
we detail our algorithm and provide examples.

\section{\label{sec:Preliminaries:-quadratures-for}Preliminaries: quadratures
for band-limited functions}

\subsection{Band-limited functions as a replacement of polynomials}

The quadratures constructed in \cite{XI-RO-YA:2001,BEY-MON:2002,RE-BE-MO:2013}
break with the conventional approach of using polynomials as the fundamental
tool in analysis and computation. The approach based on polynomial
approximations has a long tradition and leads to such notions as the
order of convergence of numerical schemes, polynomial based interpolation,
and so on. Recently, an alternative to polynomial approximations has
been developed; it turns out that constructing quadratures for band-limited
functions, e.g., exponentials $e^{ibx}$, with $|b|\le c$, where
$c$ is the bandlimit, in many cases leads to significant improvement
in performance of algorithms for interpolation, estimation and solving
partial differential equations \cite{BEY-SAN:2005,SAN-WOJ:2011,KON-ROK:2012}.

\subsection{\label{sec:Bases-for-bandlimited}Bases for band-limited functions }

It is well-known that a function whose Fourier Transform has compact
support can not have compact support itself unless it is identically
zero. On the other hand, in physics duration of all signals is finite
and their frequency response for all practical purposes is also band-limited.
Thus, it is important to identify classes of functions which are essentially
time and frequency limited. Towards this end, it is natural to analyze
an operator whose effect on a function is to truncate it both in the
original and the Fourier domains. Indeed, this has been the topic
of a series of seminal papers by Slepian, Landau and Pollak, \cite{SLE-POL:1961,LAN-POL:1961,LAN-POL:1962,SLEPIA:1964,SLEPIA:1965,SLEPIA:1978,SLEPIA:1983},
where they observed (\textit{inter alia}) that the eigenfunctions
of such operator (see~(\ref{PSWFdef}) below) are the Prolate Spheroidal
Wave Functions (PSWFs) of classical Mathematical Physics.

While periodic band-limited functions may be expanded into Fourier
series, neither the Fourier series nor the Fourier integral may be
used efficiently for non-periodic functions on \textit{intervals}.
This motivates us to consider a class of functions (not necessarily
periodic) admitting a representation via exponentials $\left\{ e^{ibx}\right\} _{\left|b\right|\le c}$,
$x\in\left[-1,1\right]$, with a fixed parameter $c$ (bandlimit).
Following \cite{BEY-MON:2002}, let us consider the linear space of
functions 
\[
\mathcal{E}_{c}=\left\{ u\!\in\! L^{\infty}(\left[-1,1\right])\ \left|\right.\ u(x)=\sum_{k\in\mathbb{Z}}a_{k}e^{icb_{k}x}:\left\{ a_{k}\right\} {}_{k\in\mathbb{Z}}\!\in\! l^{1},\ b_{k}\!\in\left[-1,1\right]\right\} .
\]
Given a finite accuracy $\epsilon$, we represent the functions in
$\mathcal{E}_{c}$ by a fixed set of exponentials $\left\{ e^{ic\tau_{k}x}\right\} _{k=1}^{M}$,
where $M$ is as small as possible. It turns out that by finding quadrature
nodes $\left\{ \tau_{k}\right\} {}_{k=1}^{M}$ and weights $\left\{ w_{k}\right\} {}_{k=1}^{M}$
for exponentials with bandlimit $2c$ and accuracy $\epsilon^{2}$,
we in fact obtain (with accuracy $\epsilon$) a basis for $\mathcal{E}_{c}$
with bandlimit $c$ \cite{BEY-MON:2002}. 

The generalized Gaussian quadratures for exponentials are constructed
in \cite{BEY-MON:2002} (see \cite{XI-RO-YA:2001} and \cite{RE-BE-MO:2013}
for different constructions), which we summarize as 
\begin{lem}
\label{lem:Quadratures}For $c>0$ and any $\epsilon>0$, there exist
nodes $-1<\tau_{1}<\tau_{2}<\dots<\tau_{M}<1$ and corresponding weights
$w_{k}>0$, such that for any $x\in\left[-1,1\right]$, 
\begin{equation}
\left|\int_{-1}^{1}e^{ictx}\ dt-\sum_{k=1}^{M}w_{k}e^{ic\tau_{k}x}\right|<\epsilon,\label{blfquad}
\end{equation}
where the number of nodes, $M$, is (nearly) optimal. The nodes and
weights maintain the natural symmetry, $\tau_{k}=-\tau_{M-k+1}$ and
$w_{k}=w_{M-k+1}$.\end{lem}
\begin{rem}
The construction in \cite{BEY-MON:2002} is more general and yields
quadratures for band-limited exponentials integrated with a weight
function. If the weight function is $1$ as in Lemma~\ref{lem:Quadratures},
then the approach in \cite{BEY-MON:2002} identifies the nodes of
the generalized Gaussian quadratures in (\ref{blfquad}) as zeros
of the \textit{Discrete} Prolate Spheroidal Wave Functions (DPSWFs)
\cite{SLEPIA:1978}, corresponding to small eigenvalues.
\end{rem}
Next we consider band-limited functions, 
\[
\mathcal{B}_{c}=\{f\in L^{2}(\mathbb{R})\,\left|\right.\hat{f}(\omega)=0\,\,\mbox{for}\,\,\left|\omega\right|\ge c\},
\]
and briefly summarize some of the results in \cite{SLE-POL:1961,LAN-POL:1961,LAN-POL:1962,SLEPIA:1964,SLEPIA:1983}.
Let us define the operator $F_{c}:\,\, L^{2}\left[-1,1\right]\rightarrow L^{2}\left[-1,1\right]$,
\begin{equation}
F_{c}(\psi)(\omega)=\int_{-1}^{1}e^{icx\omega}\psi(x)dx,\label{PSWFdef}
\end{equation}
where $c>0$ is the bandlimit. We also consider the operator $Q_{c}=\frac{c}{2\pi}F_{c}^{*}F_{c}$,
\begin{equation}
Q_{c}(\psi)(y)=\frac{1}{\pi}\int_{-1}^{1}\frac{\sin(c(y-x))}{y-x}\psi(x)\,\, dx.\label{Qdef}
\end{equation}
The eigenfunctions $\psi_{0}^{c},\psi_{1}^{c},\psi_{2}^{c},\cdots$
of $Q_{c}$ coincide with those of $F_{c}$, and the eigenvalues $\mu_{j}$
of $Q_{c}$ are related to the eigenvalues $\lambda_{j}$ of $F_{c}$
as 
\begin{equation}
\mu_{j}=\frac{c}{2\pi}|\lambda_{j}|^{2},\,\,\,\,\, j=0,1,2,\dots.\label{mula}
\end{equation}
While all $\mu_{j}<1$, $j=0,1,\dots$, for large $c$ the first approximately
$2c/\pi$ eigenvalues $\mu_{j}$ are close to $1$. They are followed
by $\mathcal{O}(\log c)$ eigenvalues which decay exponentially fast
forming a transition region; the rest of the eigenvalues $\mu_{j}$
are very close to zero.

The key result in \cite{SLE-POL:1961} states that there exists a
strictly increasing sequence of real numbers $\gamma_{0}<\gamma_{1}\dots$,
such that $\psi_{j}^{c}$ are eigenfunctions of the differential operator,
\begin{equation}
L_{c}\psi_{j}^{c}\equiv\left(-(1-x^{2})\,\frac{d^{2}}{dx^{2}}\,+\,2x\frac{d}{dx}\,+\, c^{2}x^{2}\right)\psi_{j}^{c}(x)\,=\,\gamma_{j}\psi_{j}^{c}(x)\,.\label{eq:diffEQforProlates}
\end{equation}
The eigenfunctions of $L_{c}$ have been known as the angular Prolate
Spheroidal Wave Functions (PSWF) before the connection with (\ref{PSWFdef})
was discovered in \cite{SLE-POL:1961} by demonstrating that $L_{c}$
and $F_{c}$ commute. We note that if $c\to0$, then it follows from
(\ref{eq:diffEQforProlates}) that, in this limit, $\psi_{j}^{c}$
become the Legendre polynomials. In many respects, PSWFs are strikingly
similar to orthogonal polynomials; they are orthonormal, constitute
a Chebychev system, and admit a version of Gaussian quadratures \cite{XI-RO-YA:2001}. 

Since the space $\mathcal{E}_{c}$ is dense in $\mathcal{B}_{c}$
(and vice versa) \cite{BEY-MON:2002}, we note that the quadratures
in \cite{XI-RO-YA:2001} may potentially be used for the purposes
of this paper as well (the nodes of the quadratures in \cite{XI-RO-YA:2001}
and those used in this paper are close but are not identical). Importantly,
given accuracy $\epsilon$, the functions $\psi_{0}^{c},\psi_{1}^{c},\psi_{2}^{c},\cdots,\psi_{M-1}^{c}$
may be used as a basis for interpolation on the interval $\left[-1,1\right]$
with $\tau_{1},\tau_{2},\cdots,\tau_{M}$ as the interpolation nodes,
provided that these are quadrature nodes constructed for the bandlimit
$2c$ and accuracy $\epsilon^{2}$. Given functions $\psi_{0}^{c},\psi_{1}^{c},\psi_{2}^{c},\cdots,\psi_{M-1}^{c}$,
we can construct an analogue of the Lagrange interpolating polynomials,
$R_{k}^{c}(x)=\sum_{j=0}^{M-1}\alpha_{kj}\psi_{j}^{c}(x)$, $x\in\left[-1,1\right]$,
by solving 
\begin{equation}
\delta_{kl}=R_{k}^{c}(\tau_{l})=\sum_{j=0}^{M-1}\alpha_{kj}\psi_{j}^{c}(\tau_{l})\label{eq:ExactInterpolatingFnc}
\end{equation}
for the coefficients $\alpha_{kj}$. The matrix $\psi_{j}^{c}(\tau_{l})$
in (\ref{eq:ExactInterpolatingFnc}) is well conditioned.

A well-known problem associated with the numerical use of orthogonal
polynomials is concentration of their roots near the ends of the interval.
Let us consider the ratio 
\begin{equation}
r(M,\epsilon)=\frac{\tau_{2}-\tau_{1}}{\tau_{\lfloor M/2\rfloor}-\tau_{\lfloor M/2\rfloor-1}},\label{node-ratio}
\end{equation}
where ``$\lfloor M/2\rfloor$'' denotes the least integer part,
and look at it as a function of $M$. Observing that the distance
between nodes of Gaussian quadratures for exponentials changes monotonically
from the middle of an interval toward its end points, and that the
smallest distance occurs between the nodes closest to the end point,
the ratio (\ref{node-ratio}) may be used as a measure of node accumulation.
For example, the distance between the nodes near the end points of
the standard Gaussian quadratures for polynomials decreases as $\mathcal{O}(1/M^{2})$,
so that we have $r(M,\epsilon)=\mathcal{O}(1/M)$, where $M$ is the
number of nodes. In Figure~\ref{ratio-plot} we illustrate the behavior
of $r(M,\epsilon)$ for the nodes of quadratures for band-limited
exponentials. This ratio approaches a constant that depends on the
accuracy $\epsilon$ but does not depend on the number of nodes.

Another important property of quadratures for exponentials emerges
if we compare the critical sampling rate of a smooth \textit{periodic}
function, to that of smooth \textit{non-periodic} function defined
on an interval. Considering bandlimit $c$ as a function of the number
of nodes, $M$, and the desired accuracy $\epsilon$, we observe that
the oversampling factor, 
\[
\alpha(M,\epsilon)=\frac{\pi M}{c(M,\epsilon)}>1,
\]
approaches $1$ for large M. We recall that in the case of the Gaussian
quadratures for polynomials, this oversampling factor approaches $\frac{\pi}{2}$
rather than $1$ (see e.g. \cite{GOT-ORS:1977}). 

\begin{figure}
\begin{centering}
\includegraphics[scale=1.2,bb = 0 0 300 200]{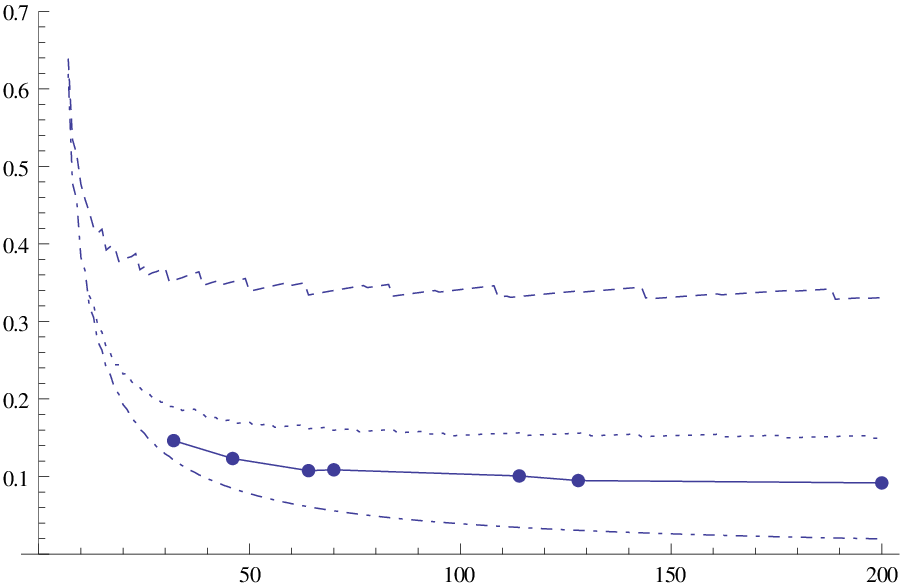}
\par\end{centering}

\caption{\label{ratio-plot}The ratio $r(M,\epsilon)$ in (\ref{node-ratio})
as a function of the number of nodes $M$ and interpolation accuracy
$\epsilon\approx10^{-3.5}$ (top curve, dashed), $\epsilon\approx10^{-8.5}$
(middle curve, dotted) and $\epsilon\approx10^{-13}$ (middle curve,
solid). The dots on the solid curve indicate the number of nodes of
quadratures used in our numerical experiments. The bottom curve shows
this ratio for the Gauss-Legendre nodes.}
\end{figure}

\subsection{\label{sub:Interpolating-Bandlimiting-functions}Interpolating bases
for band-limited functions }

A basis of interpolating band-limited functions for the bandlimit
$c$ and accuracy $\epsilon$ plays the same role in the derivation
of a system of nonlinear equations for solving ODEs as the bases of
Lagrange interpolating polynomials defined on the Gauss-Legendre nodes.
While (\ref{eq:ExactInterpolatingFnc}) relies on available solutions
of the differential equation (\ref{eq:diffEQforProlates}), interpolating
basis functions may also be obtained by solving the integral equation
(\ref{PSWFdef}) (see \cite{BEY-MON:2002,RE-BE-MO:2013}).

We start by first constructing a quadrature for the bandlimit $2c>0$
and accuracy threshold $\epsilon^{2}>0$, yielding $M$ nodes $\left\{ \tau_{m}\right\} _{m=1}^{M}$
and weights $\left\{ w_{m}\right\} _{m=1}^{M}$. For the inner product
of two functions $f,g\in\mathcal{E}_{c}$, we have
\[
\left|\int_{-1}^{1}f(t)g(t)dt-\sum_{m=1}^{M}w_{m}f(\tau_{m})g(\tau_{m})\right|\le\epsilon^{2}.
\]
Following \cite{BEY-MON:2002}, we discretize (\ref{PSWFdef}) using
nodes $\left\{ \tau_{m}\right\} _{m=1}^{M}$ and weights $\left\{ w_{m}\right\} _{m=1}^{M}$
and obtain an algebraic eigenvalue problem, 
\begin{equation}
\sum_{l=1}^{M}w_{l}e^{ic\tau_{m}\tau_{l}}\mathbf{\Psi}_{j}(\tau_{l})=\eta_{j}\mathbf{\Psi}_{j}(\tau_{m}).\label{eq:APSWFdefa}
\end{equation}
The approximate PSWFs on $[-1,1]$ are then defined consistent with
(\ref{PSWFdef}) as 
\begin{equation}
\Psi_{j}(x)=\frac{1}{\eta_{j}}\sum_{l=1}^{M}w_{l}e^{icx\tau_{l}}\mathbf{\Psi}_{j}(\tau_{l}),\label{eq:APSWFdefb}
\end{equation}
where $\eta_{j}$ are the eigenvalues and $\mathbf{\Psi}_{j}(\tau_{l})$
the eigenvectors in (\ref{eq:APSWFdefa}). Following \cite{BEY-MON:2002},
we then define the interpolating basis for band-limited functions
as 

\begin{equation}
R_{k}(x)=\sum_{l=1}^{M}r_{kl}e^{ic\tau_{l}x},\,\,\,\, k=1,\ldots,M,\label{eq:rkx}
\end{equation}
where 
\begin{equation}
r_{kl}=\sum_{j=1}^{M}w_{k}\mathbf{\Psi}_{j}(\tau_{k})\frac{1}{\eta_{j}}\mathbf{\Psi}_{j}(\tau_{l})w_{l}.\label{eq:CoefsForInterpolatingBases}
\end{equation}
 It is shown in \cite{BEY-MON:2002} that the functions $R_{k}(x)$
are interpolating, $R_{k}(\tau_{l})=\delta_{kl}$.

\section{BLC-IRK method\label{sec:BQC-IRK-method}}

\subsection{\label{sec:Discretization-of-Picard}Discretization of Picard integral
equation}

We consider the initial value problem for a system of ODEs, 
\[
\mathbf{y}'=\mathbf{f}(t,\mathbf{y}),\,\,\,\,\,\mathbf{y}(0)=\mathbf{y_{0}},
\]
or, equivalently, 
\begin{equation}
\mathbf{y}(t)=\mathbf{y_{0}}+\int_{0}^{t}\mathbf{f}(s,\mathbf{y}(s))\ ds.\label{eq:intEq}
\end{equation}
It is sufficient to discretize (\ref{eq:intEq}) on the interval $\left[0,t\right]$
since, by shifting the time variable, the initial condition may always
be set at $t=0$. We require 
\[
\mathbf{y}'(t\tau_{j})=\mathbf{f}(t\tau_{j},\mathbf{y}(t\tau_{j})),\,\,\, j=1,\dots,M,
\]
where $\{\tau_{j}\}_{j=1}^{M}$ are Gaussian nodes for band-limited
exponentials on $[0,1]$ (constructed for an appropriate bandlimit
$c$ and accuracy $\epsilon$). We approximate 
\begin{equation}
\|\mathbf{f}(t\tau,\mathbf{y}(t\tau))-\sum_{j=1}^{M}\mathbf{f}(t\tau_{j},\mathbf{y}(t\tau_{j}))R_{j}(\tau)\|\leq\epsilon,\,\,\,\,\tau\in\left[0,1\right]\label{eq:fcnApp}
\end{equation}
where $R_{j}(\tau)$ are interpolating basis functions associated
with these quadratures and briefly described in Section~\ref{sub:Interpolating-Bandlimiting-functions}
(see \cite{BEY-MON:2002,BEY-SAN:2005} for details). Using (\ref{eq:fcnApp}),
we replace $\mathbf{f}$ in (\ref{eq:intEq}) and evaluate $\mathbf{y}(t\tau)$
at the quadrature nodes yielding a nonlinear system, 
\begin{eqnarray}
\mathbf{y}(t\tau_{k}) & = & \mathbf{y_{0}}+\sum_{j=1}^{M}\mathbf{f}(t\tau_{j},\mathbf{y}(t\tau_{j}))\int_{0}^{\tau_{k}}R_{j}(s)ds\label{eq:intEqDis}\\
 & = & \mathbf{y_{0}}+\sum_{j=1}^{M}S_{kj}\mathbf{f}(t\tau_{j},\mathbf{y}(t\tau_{j})),\nonumber 
\end{eqnarray}
 where $S_{kj}=\int_{0}^{\tau_{k}}R_{j}(s)ds$ is the integration
matrix and $k=1,\dots M$. After solving for $\{\mathbf{y}(t\tau_{j})\}_{j=1}^{M}$,
we have from (\ref{eq:intEq}) 
\begin{equation}
\mathbf{y}(t)=\mathbf{y_{0}}+\sum_{j=1}^{M}w_{j}\mathbf{f}(t\tau_{j},\mathbf{y}(t\tau_{j})),\label{ImplicitRK}
\end{equation}
where $\{w_{j}\}_{j=1}^{M}$ are the quadrature weights. The result
is an implicit Runge-Kutta method (IRK) where the usual Gauss-Legendre
quadratures are replaced by Gaussian quadratures for band-limited
exponentials. 

The nodes, weights, and the entries of the integration matrix are
typically organized in the Butcher tableau,

\begin{center}
\begin{tabular}{c|c}
$\tau$  & $S$\tabularnewline
\hline 
 & $w^{t}$\tabularnewline
\end{tabular}. 
\par\end{center}

Unlike in the standard IRK method based on Gauss-Legendre quadratures,
we solve (\ref{eq:intEqDis}) on a time interval containing a large
number of quadrature nodes, since these nodes do not concentrate excessively
near the end points. This implies that the interval $\left[0,t\right]$
may be selected to be large in comparison with the usual choices in
RK methods.

\subsection{Exact Linear Part}

In many problems (including that of orbit computations in astrodynamics),
the right hand side of the ODE, $\mathbf{f}(t,\mathbf{y})$, may be
split into a linear and nonlinear part, 
\[
\mathbf{f}(t,\mathbf{y}(t))=\mathbf{L}\mathbf{y}(t)+\mathbf{g}(t,\mathbf{y}(t)),
\]
 so that the integral equation (\ref{eq:intEq}) may be written as
\begin{equation}
\mathbf{y}(t)=e^{tL}\mathbf{y_{0}}+\int_{0}^{t}e^{(t-s)\mathbf{L}}\mathbf{g}(s,\mathbf{y}(s))\ ds.\label{eq:intEqELP}
\end{equation}
If the operator $e^{t\mathbf{L}}$ can be computed efficiently, this
formulation leads to savings when solving the integral equation iteratively. 

We discretize (\ref{eq:intEqELP}) by using (\ref{eq:fcnApp}) and
obtain 
\begin{eqnarray}
\mathbf{y}(t\tau_{k}) & = & e^{t\tau_{k}\mathbf{L}}\mathbf{y_{0}}+\sum_{j=1}^{M}e^{t(\tau_{k}-\tau_{j})\mathbf{L}}\mathbf{g}(t\tau_{j},\mathbf{y}(t\tau_{j}))\int_{0}^{\tau_{k}}R_{j}(s)ds\nonumber \\
 & = & e^{t\tau_{k}\mathbf{L}}\mathbf{y_{0}}+\sum_{j=1}^{M}S_{kj}e^{t(\tau_{k}-\tau_{j})\mathbf{L}}\mathbf{g}(t\tau_{j},\mathbf{y}(t\tau_{j}))\label{eq:intEqDis2}
\end{eqnarray}
 where $S_{kj}=\int_{0}^{\tau_{k}}R_{j}(s)ds$. We note that (\ref{eq:intEqDis})
is a special case of (\ref{eq:intEqDis2}) with $\mathbf{L}=0$ and
$\mathbf{g}=\mathbf{f}$.

\subsection{\label{sec:Symplectic-integrators}Symplectic integrators}

Following \cite{SANZ-S:1988}, let us introduce matrix $\mathcal{M}=\{m_{kj}\}_{k,j=1}^{M}$
for an $M$-stage IRK scheme, 
\begin{equation}
m_{kj}=w_{k}S_{kj}+w_{j}S_{jk}-w_{k}w_{j},\label{Matrix-M}
\end{equation}
 where the weights $w=\{w_{k}\}_{k=1}^{M}$ and the integration matrix
$S=\{S_{kj}\}_{k,j=1}^{M}$ define the Butcher's tableau for the method.

\noindent It is shown in \cite{SANZ-S:1988} that 
\begin{thm}
If matrix $\mathcal{M}=0$ in (\ref{Matrix-M}), then an $M$-stage
IRK scheme is symplectic. 
\end{thm}
This condition, $\mathcal{M}=0$, is satisfied for the Gauss-Legendre
RK methods, see e.g. \cite{DEK-VER:1984,SANZ-S:1988}. We enforce
this condition for BLC-IRK method by an $\mathcal{O}(\epsilon^{2})$
modification of the weights and of the integration matrix. For convenience,
in what follows, we consider the band-limited exponentials and integration
matrix on the interval $\left[-1,1\right]$ rather than on the interval
$\left[0,1\right]$ usually used for ODEs. 
\begin{prop}
\label{prop:Let--and}Let $\{\tau_{j}\}_{j=1}^{M}$ and $\{w_{j}\}_{j=1}^{M}$
be quadrature nodes and weights for the bandlimit $2c$ and accuracy
$\epsilon^{2}$. Consider interpolating basis functions on these quadrature
nodes, $R_{k}(\tau)$, $R_{k}(\tau_{j})=\delta_{kj}$, $k,j=1,\dots,M$,
and define $F_{k}(\tau)=\int_{-1}^{\tau}R_{k}(s)\ ds$. Then we have

\begin{equation}
\left|\int_{-1}^{1}F_{j}(\tau)F'_{k}(\tau)\ d\tau-\sum_{l=1}^{M}w_{l}F_{j}(\tau_{l})F'_{k}(\tau_{l})\right|<\epsilon^{2}\label{intMatProof3}
\end{equation}
or 
\[
\left|\int_{-1}^{1}\left(\int_{-1}^{\tau}R_{j}(s)\ ds\right)\ R_{k}(\tau)\ d\tau-w_{k}\int_{-1}^{\tau_{k}}R_{j}(s)\ ds\right|<\epsilon^{2},
\]
and 
\begin{equation}
\left|\int_{-1}^{1}R_{k}(\tau)\ d\tau-\sum_{l=1}^{M}w_{l}R_{k}(\tau_{l})\right|<\epsilon^{2},\label{eq:IntegralOfInterpolFnc}
\end{equation}
or 
\[
\left|\int_{-1}^{1}R_{k}(s)\ ds-w_{k}\right|<\epsilon^{2}.
\]
\end{prop}
\begin{proof}
The relations in (\ref{intMatProof3}) and (\ref{eq:IntegralOfInterpolFnc})
is the property of the quadrature, since the bandlimit of the product
$F_{j}(\tau)F'_{k}(\tau)$ is less or equal to $2c$ and that of $R_{k}(\tau)$
is less or equal to $c$. Due to the interpolating property of $R_{k}(\tau)$,
we have
\begin{equation}
\sum_{l=1}^{M}w_{l}F_{j}(\tau_{l})F'_{k}(\tau_{l})=\sum_{l=1}^{M}\left(\int_{-1}^{\tau_{l}}R_{j}(s)\ ds\right)\ w_{l}R_{k}(\tau_{l})=w_{k}\int_{-1}^{\tau_{k}}R_{j}(s)\ ds\label{intMatProof1}
\end{equation}
and
\[
\sum_{l=1}^{M}w_{l}R_{k}(\tau_{l})=w_{k}
\]
Also, by definition, 
\[
\int_{-1}^{1}F_{j}(\tau)F'_{k}(\tau)\ d\tau=\int_{-1}^{1}\left(\int_{0}^{\tau}R_{j}(s)\ ds\right)\ R_{k}(\tau)\ d\tau,
\]
and the result follows.\end{proof}
\begin{thm}
\label{bandlimiteSymplectic} Let $\{\tau_{j}\}_{j=1}^{M}$ be quadrature
nodes of the quadrature for the bandlimit $2c$ and accuracy $\epsilon^{2}$
and $R_{k}(\tau)$, $R_{k}(\tau_{j})=\delta_{kj}$, $k,j=1,\dots,M$,
the corresponding interpolating basis. Let us define weights for the
quadrature as 
\begin{equation}
w_{k}=\int_{-1}^{1}R_{k}(\tau)d\tau\label{eq:DefWeights}
\end{equation}
 and the integration matrix as 
\begin{equation}
S_{kj}=\frac{\int_{-1}^{1}\left(\int_{-1}^{\tau}R_{j}(s)\ ds\right)\ R_{k}(\tau)\ d\tau}{w_{k}},\,\,\, k,j=1,\dots,M.\label{eq:IntegrationMatrixDef}
\end{equation}
 Then 
\begin{equation}
w_{k}S_{kj}+w_{j}S_{jk}-w_{k}w_{j}=0,\label{eq:symplectic condition}
\end{equation}
 and the implicit scheme using these nodes and weights is symplectic. \end{thm}
\begin{proof}
Using Proposition~\ref{prop:Let--and}, we observe that the weights
defined in (\ref{eq:DefWeights}) are the same (up to accuracy $\epsilon^{2}$)
as those of the quadrature. The result follows by setting $F_{k}(\tau)=\int_{-1}^{\tau}R_{k}(\tau)\ d\tau$,
$F'_{k}(\tau)=R_{k}(\tau)$ and integrating by parts to obtain 
\[
\begin{split}w_{k}S_{kj}+w_{j}S_{jk}-w_{k}w_{j} & =\int_{-1}^{1}F_{j}(\tau)F'_{k}(\tau)\ d\tau+\int_{-1}^{1}F_{k}(\tau)F'_{j}(\tau)\ d\tau-w_{k}w_{j}\\
 & =F_{j}(1)F_{k}(1)-w_{k}w_{j}.
\end{split}
\]
 By the definition of the weights, we have $F_{k}(1)=w_{k}$ and,
hence, $F_{j}(1)F_{k}(1)-w_{k}w_{j}=0$. 
\end{proof}

\subsection{Construction of the integration matrix}

There are at least three approaches to compute the integration matrix.
Two of them, presented in the Appendix, rely on Theorem~\ref{bandlimiteSymplectic}
and differ in the construction of interpolating basis functions. In
what appears to be a simpler approach, the integration matrix may
also be obtained without computing interpolating basis functions explicitly
and, instead, using a collocation condition derived below together
with the symplectic condition (\ref{eq:symplectic condition}). 

We require that our method accurately solves the test problems 

\[
y'=ic\tau_{m}y,\,\,\,\,\, y(-1)=e^{-ic\tau_{m}},\,\,\, m=1,\dots,M,
\]
on the interval $\left[-1,1\right]$, where $\tau_{m}$ are the nodes
of the quadrature. Specifically, given solutions of these test problems,
$y_{m}(t)=e^{ic\tau_{m}t}$, we require that (\ref{eq:intEqDis})
holds at the nodes $t=\tau_{k}$ with accuracy $\epsilon$, 
\begin{equation}
\left|\frac{e^{ic\tau_{m}\tau_{k}}-e^{-ic\tau_{m}}}{ic\tau_{m}}-\sum_{j=1}^{M}S_{kj}e^{ic\tau_{m}\tau_{j}}\right|\le\epsilon,\,\,\,\, m,k=1,\dots,M.\label{eq:TestProblemEq}
\end{equation}
We then obtain the integration matrix as the solution of (\ref{eq:symplectic condition})
satisfying an approximate collocation condition (\ref{eq:TestProblemEq}).

We proceed by observing that (\ref{eq:symplectic condition}) suggests
that the integration matrix can be split into symmetric and antisymmetric
part. Defining the symmetric part of the integration matrix as 
\begin{equation}
T_{kj}=\frac{w_{k}w_{j}}{w_{k}+w_{j}},\label{eq:sym part}
\end{equation}
we set 
\begin{equation}
S_{kj}=T_{kj}+A_{kj}w_{j},\label{eq:decomposition into sym-antisym}
\end{equation}
and observe that it follows from (\ref{eq:symplectic condition})
that $A_{kj}$ is antisymmetric,

\[
A_{kj}+A_{jk}=0.
\]
Using (\ref{eq:decomposition into sym-antisym}) and casting (\ref{eq:TestProblemEq})
as an equality, we obtain equations for the matrix entries $A_{kj}$,
\begin{equation}
\sum_{j=1}^{M}A_{kj}w_{j}e^{ic\tau_{m}\tau_{j}}=\frac{e^{ic\tau_{m}\tau_{k}}-e^{-ic\tau_{m}}}{ic\tau_{m}}-\sum_{j=1}^{M}T_{kj}e^{ic\tau_{m}\tau_{j}},\,\,\,\, m,k=1,\dots,M.\label{eq:TestProblemEq-1}
\end{equation}
Splitting the real and imaginary parts of the right hand side, 
\[
\frac{e^{ic\tau_{m}\tau_{k}}-e^{-ic\tau_{m}}}{ic\tau_{m}}-\sum_{j=1}^{M}T_{kj}e^{ic\tau_{m}\tau_{j}}=u_{km}+iv_{km},
\]
we obtain 
\[
u_{km}=\left(\tau_{k}+1\right)\mbox{sinc}\left(c\tau_{m}(\tau_{k}+1)/2\right)\cos\left(c\tau_{m}(\tau_{k}-1)/2\right)-\sum_{j=1}^{M}T_{kj}\cos\left(c\tau_{m}\tau_{j}\right)
\]
and, since $T_{kj}=T_{k\left(M-j+1\right)}$due to the symmetry of
the weights, we arrive at 
\[
v_{km}=\left(\tau_{k}+1\right)\mbox{sinc}\left(c\tau_{m}(\tau_{k}+1)/2\right)\sin\left(c\tau_{m}(\tau_{k}-1)/2\right).
\]
We also have 

\[
u_{km}=\sum_{j=1}^{M}A_{kj}w_{j}\cos\left(c\tau_{m}\tau_{j}\right),\,\,\,\,\,\,\, v_{km}=\sum_{j=1}^{M}A_{kj}w_{j}\sin\left(c\tau_{m}\tau_{j}\right).
\]
Since matrices $\cos\left(c\tau_{m}\tau_{j}\right)$ and $\sin\left(c\tau_{m}\tau_{j}\right)$
are rank deficient, we choose to combine these equations 
\begin{equation}
u_{km}+v_{km}=\sum_{j=1}^{M}A_{kj}w_{j}\left(\cos\left(c\tau_{m}\tau_{j}\right)+\sin\left(c\tau_{m}\tau_{j}\right)\right).\label{eq:EquationForAntiSymPart}
\end{equation}
The number of unknowns in (\ref{eq:EquationForAntiSymPart}) is $M(M-1)/2$
since the matrix $A$ is antisymmetric. Instead of imposing additional
conditions due to antisymmetry of $A$, we proceed by solving (\ref{eq:EquationForAntiSymPart})
using quadruple precision (since this system is ill-conditioned).
We find matrix $\tilde{A}$ and discover that, while $S_{kj}=T_{kj}+\tilde{A}_{kj}w_{j}$
makes (\ref{eq:TestProblemEq}) into an equality, the matrix $\tilde{A}$
is not antisymmetric. We then enforce anti-symmetry by setting $A_{kj}=-A_{jk}=\left(\tilde{A}_{kj}-\tilde{A}_{jk}\right)/2$
and $S_{kj}=T_{kj}+A_{kj}w_{j}$. We then verify that the matrix $S$
satisfies the inequality (\ref{eq:TestProblemEq}).
\begin{rem}
\label{rem:Exact?}The fact that integration matrix satisfies (\ref{eq:symplectic condition})
and the inequality (\ref{eq:TestProblemEq}) indicates that, perhaps
by a slight modification of nodes and weights of the quadrature, it
might be possible to satisfy (\ref{eq:symplectic condition}) and
(\ref{eq:TestProblemEq}) with $\epsilon=0$. 
\end{rem}

\subsection{\label{sec:A-Stability-of-the}A-stability of the BLC-IRK method}

As shown in e.g. \cite[Section 4.3]{ISERLE:1996}, in order to ascertain
stability of an IRK method, it is sufficient to consider the rational
function
\begin{equation}
r(z)=1+z\mathbf{w}^{t}(I-zS)^{-1}\mathbf{1},\label{eq:EqForA-stability}
\end{equation}
where $S$ is the integration matrix, $\mathbf{w}$ is a vector of
weights and $\mathbf{1}$ is a vector with all entries set to $1$,
and verify that $\left|r(z)\right|\le1$ in the left half of the complex
plane, $\mathcal{R}e\left(z\right)\le0$. This function is an approximation
of the solution $e^{zt}$ at $t=1$ of the test problem 
\[
y'=zy,\,\,\,\,\, y(0)=1
\]
computed via (\ref{eq:intEqDis}) and (\ref{ImplicitRK}) on the interval
$\left[0,1\right]$. If all poles of $r(z)$ have a positive real
part, then it is sufficient to verify this inequality only on the
imaginary axis, $z=iy$, $y\in\mathbb{R}$. In fact, it may be possible
to show that $r(z)$ is unimodular on imaginary axis, $\left|r(iy)\right|=1$,
for $y\in\mathbb{R}$. Implicit Runge-Kutta methods based on Gauss-Legendre
nodes are A-stable (see e.g \cite{ISERLE:1996}) and, indeed, for
these methods $r(z)$ is unimodular on imaginary axis. 

Given an $M\times M$ matrix $S$ with $M_{1}$ complex eigenvalues
and $M_{2}$ real eigenvalues implies that the function $r(z)$ in
(\ref{eq:EqForA-stability}) has $2M_{1}+M_{2}=M$ poles. If this
function is unimodular on the imaginary axis then it is easy to show
that it has a particular form, 
\begin{equation}
r(z)=\prod_{k=1}^{M_{1}}\frac{z+\overline{\lambda}_{k}^{-1}}{z-\lambda_{k}^{-1}}\frac{z+\lambda_{k}^{-1}}{z-\overline{\lambda}_{k}^{-1}}\prod_{k'=1}^{M_{2}}\frac{z+\lambda_{k'}^{-1}}{z-\lambda_{k'}^{-1}}.\label{eq:r(z) as rational fnc}
\end{equation}
Currently, we do not have an analytic proof of A-stability of BLC-IRK
method; instead we verify (\ref{eq:r(z) as rational fnc}) numerically.
We compute eigenvalues of the integration matrix to obtain the poles
of $r(z)$ and check that all eigenvalues have a positive real part
separated from zero. For example, the integration matrix for the BLC-IRK
method with $64$ nodes (bandlimit $c=17\pi$) has all eigenvalues
with real part larger than $0.7\cdot10^{-3}$(see Figure~\ref{fig:Eigenvalues-of-the}).
One way to check that $r(z)$ has the form (\ref{eq:r(z) as rational fnc})
is to compute $r(-\overline{\lambda}_{k}^{-1})$ for complex valued
and $r(-\lambda_{k}^{-1})$ for real valued eigenvalues in order to
observe if these are its zeros. In fact, it is the case with high
(quadruple) precision. 

One can argue heuristically that since a rational function with $M$
poles has at most $2M$ real parameters (since matrix $S$ is real
its eigenvalues appear in complex conjugate pairs) and since, by construction,
$r(iy)$ for $\left|y\right|\le c$ is an accurate approximation to
$e^{iy}$ (which is obviously unimodular), $r(z)$ is then unimodular.
It remains to show it rigorously; a possible proof may depend on demonstrating
a conjecture in Remark~\ref{rem:Exact?}.

\begin{figure}
\vskip 1cm
\begin{centering}
\includegraphics[scale=0.85,bb = 0 0 300 200]{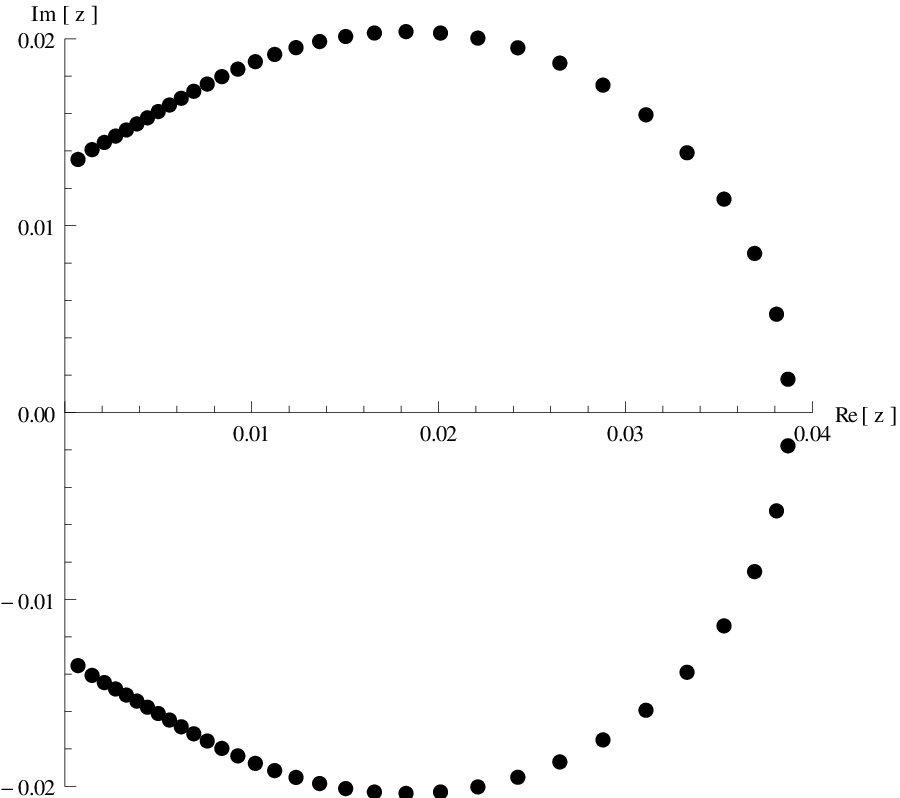}
\par\end{centering}
\vskip 1cm
\begin{centering}
\includegraphics[scale=0.85,bb = 0 0 300 200]{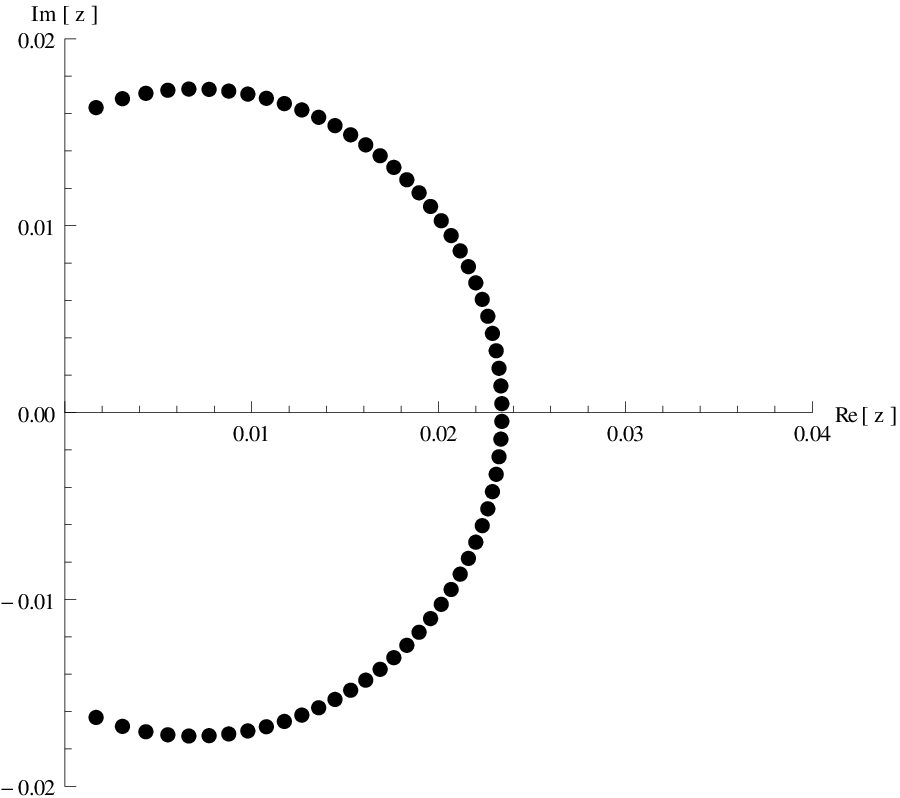}
\par\end{centering}

\caption{\label{fig:Eigenvalues-of-the}Eigenvalues (computed using quadruple
precision) of the integration matrix for BLC-IRK scheme with $64$
nodes corresponding to the bandlimit $17\pi$ and, for comparison,
eigenvalues of the integration matrix of the standard IRK scheme $64$
Gauss-Legendre nodes. }

\end{figure}

\section{Applications\label{sec:Applications}}

\subsection{Algorithm\label{sec:alg} }

We use a (modified) fixed point iteration to solve (\ref{eq:intEqDis2}).
These equations are formulated on a large time interval in comparison
with the polynomial-based IRK schemes since we do not have to deal
with the excessive concentration of nodes near the end points. Thus,
the only constraint on the size of the interval is the requirement
that the (standard) fixed point iteration for (\ref{eq:intEqDis2})
converges . 

Let $N_{it}$ denote the number of iterations, which can either be
set to a fixed number or be determined adaptively. Labeling the intermediate
solutions in the iteration scheme as $\mathbf{y}^{(n)},\ n=1,\ldots,N_{it}$,
we have 
\begin{enumerate}
\item Initialize $y^{(1)}(t\tau_{m})=\mathbf{y_{0}},\ m=1,\ldots,M$. 
\item \textbf{For} $n=1,\ldots,N_{it}$\\
\textbf{For} $k=1,\dots,M$\\

\begin{enumerate}
\item Update the solution at the node $k$:\\
 \\
 $\mathbf{y}^{(n)}(t\tau_{k})=e^{t\tau_{k}\mathbf{L}}\mathbf{y_{0}}+\sum_{j=1}^{M}S_{kj}\, e^{t(\tau_{k}-\tau_{j})\mathbf{L}}\mathbf{g}(t\tau_{j},\mathbf{y}^{(n)}(t\tau_{j}))$
\\

\item Update the right hand side at the node $k$: $\mathbf{g}(t\tau_{k},y^{(n)}(t\tau_{k}))$ 
\end{enumerate}
\end{enumerate}
We note that the updated value of $y^{(n)}(t\tau_{k})$ is used in
the computation at the next node $\tau_{k+1}$ within the same iteration
$n$. This modification of the standard fixed point iteration is essential
for a faster convergence.
\begin{rem}
Although we currently apply the integration matrix directly, using
a large time interval and, consequently, a large number of nodes per
interval, opens a possibility of developing fast algorithms for this
purpose. Such algorithms may be faster than the direct application
of the matrix only for a sufficiently large matrix size and are typically
less efficient than the direct method if the size is relatively small.
Since we may choose many nodes, it makes sense to ask if the integration
matrix of an BLC-IRK type method may be applied in $\mathcal{O}\left(M\right)$
or $\mathcal{O}\left(M\log M\right)$ operations rather than $\mathcal{O}\left(M^{2}\right)$.
We mention an example of an algorithm for this purpose using the partitioned
low rank (PLR) representation (as it was described in e.g., \cite{BEY-SAN:2005})
but leave open a possibility of more efficient approaches. 
\end{rem}

\subsection{Problem of Orbit Determination\label{sub:Problem-of-Orbit}}

Let us consider the spherical harmonic model of a gravitational potential
of degree $N$, 
\begin{equation}
V^{(N)}(r,\theta,\lambda)=\frac{\mu}{r}\left(1+\sum_{n=2}^{N}\left(\frac{R}{r}\right)^{-n}Y_{n}(\theta,\lambda)\right),\label{eq:GravPotent}
\end{equation}
with
\begin{equation}
Y_{n}(\theta,\lambda)=\sum_{m=0}^{n}\,\bar{P}_{n}^{m}(\sin\theta)({\bar{C}}_{nm}\cos(m\lambda)+{\bar{S}}_{nm}\sin(m\lambda)),\label{eq:SphModelCoeffs}
\end{equation}
where ${\bar{P}}_{n}^{m}$ are normalized associated Legendre functions
and ${\bar{C}}_{nm}$ and ${\bar{S}}_{nm}$ are normalized gravitational
coefficients. In case of the Earth's gravitational model, $\mu$ is
the Earth's gravitational constant and $R$ is chosen to be the Earth's
equatorial radius. Choosing the Cartesian coordinates, we write $V^{(N)}\left(\mathbf{r}\right)$,
$\mathbf{r}=\left(x,y,z\right)$, assuming that the values $V^{(N)}\left(\mathbf{r}\right)$
are evaluated via (\ref{eq:GravPotent}) by changing from the Cartesian
to the spherical coordinates, $r=\sqrt{x^{2}+y^{2}+z^{2}}$, $\theta=\arcsin(z/r)$
and $\lambda=\arctan\left(y/x\right)$.

We formulate the system of ODEs in the Cartesian coordinates and denote
the solution as $\mathbf{r}(t)=\left(x(t),y(t),z(t)\right)$. Setting
$\mathbf{G}^{(N)}\left(\mathbf{r}\right)=\nabla V^{(N)}\left(\mathbf{r}\right)$,
we consider the initial value problem
\begin{equation}
\frac{d^{2}}{dt^{2}}\mathbf{r}(t)=-\mathbf{G}^{(N)}\left(\mathbf{r}(t)\right),\,\,\,\mathbf{r}(0)=\mathbf{r}_{0}=\left(\begin{array}{c}
x_{0}\\
y_{0}\\
z_{0}
\end{array}\right),\,\,\,\mathbf{r'}(0)=\mathbf{v}_{0}=\left(\begin{array}{c}
x'_{0}\\
y'_{0}\\
z'_{0}
\end{array}\right).\label{eq:NewtonsLawForOrbits}
\end{equation}

We observe that the first few terms of the Earth's gravitational models
are large in comparison with the rest of the model terms. For example,
in EGM96 \cite{L-K-F-T-P-C-C-K-L-T-1998}, the only non-zero coefficients
for $Y_{2}(\theta,\lambda)$ are $\bar{C}_{20}$, $\bar{C}_{22}$
and $\bar{S}_{22}$, where $\bar{C}_{20}\approx-0.48\cdot10^{-3}$,
$\bar{C}_{22}\approx0.24\cdot10^{-5}$, and $\bar{S}_{22}\approx-0.14\cdot10^{-5}$,
whereas the coefficients of the terms $Y_{n}(\theta,\lambda)$ with
$n\ge3$ are less than $0.14\cdot10^{-5}$. For this reason it makes
sense to split the force as 
\[
\mathbf{G}^{(N)}\left(\mathbf{r}\right)=\mathbf{G}^{(2)}\left(\mathbf{r}\right)+\left(\mathbf{G}^{(N)}\left(\mathbf{r}\right)-\mathbf{G}^{(2)}\left(\mathbf{r}\right)\right)
\]
and use only $\mathbf{G}^{(2)}\left(\mathbf{r}\right)$ in most of
the iterations (since using the full model, $\mathbf{G}^{(N)}\left(\mathbf{r}\right)$,
may be expensive). 

We first use the gravity model of degree $N=2$ on a large portion
of an orbit (e.g., $1/2$ of a period) to solve the system of nonlinear
equations via fixed point iteration. Once the approximate solution
$\tilde{\mathbf{r}}(t)$ to 
\[
\frac{d^{2}}{dt^{2}}\tilde{\mathbf{r}}(t)=-\mathbf{G}^{(2)}\left(\tilde{\mathbf{r}}(t)\right),\,\,\,\tilde{\mathbf{r}}(0)=\mathbf{r}_{0}=\left(\begin{array}{c}
x_{0}\\
y_{0}\\
z_{0}
\end{array}\right),\,\,\,\tilde{\mathbf{r}}'(0)=\mathbf{v}_{0}=\left(\begin{array}{c}
x'_{0}\\
y'_{0}\\
z'_{0}
\end{array}\right),
\]
is obtained, we then access the full gravity model $\mathbf{G}^{(N)}\left(\tilde{\mathbf{r}}(t\tau_{j})\right)$
to evaluate the forces at the nodes $\tau_{j}$ which, by now, are
located close to their correct positions. We continue iteration (without
accessing the full gravity model again) to adjust the orbit. This
results in an essentially correct trajectory. At this point we may
(and currently do) access the full gravity model $\mathbf{G}^{(N)}$
one more time to evaluate the gravitational force and perform another
iteration. Thus, we access the full gravity model at most twice per
node while the number of nodes is substantially lower than in traditional
methods. 

Next, let us write the orbit determination problem in a form that
conforms with the algorithm in Section \ref{sec:alg}. Effectively,
we make use of the fact that system (\ref{eq:NewtonsLawForOrbits})
is of the second order. We define the six component vector
\[
\mathbf{u}(t)=\left[\begin{array}{c}
\mathbf{r}(t)\\
\mathbf{r}'(t)
\end{array}\right]=\left[\begin{array}{c}
\mathbf{r}(t)\\
\mathbf{v}(t)
\end{array}\right],
\]
where $\mathbf{r}'(t)=\mathbf{v}(t)$ is the velocity, and the matrix
\[
\mathbf{L}=\left(\begin{array}{cc}
\mathbf{0} & \mathbf{I}\\
\mathbf{0} & \mathbf{0}
\end{array}\right),
\]
where $\mathbf{I}$ is $3\times3$ identity matrix. We have 
\begin{equation}
\frac{d}{dt}\left[\begin{array}{c}
\mathbf{r}(t)\\
\mathbf{v}(t)
\end{array}\right]=\mathbf{L}\left[\begin{array}{c}
\mathbf{r}(t)\\
\mathbf{v}(t)
\end{array}\right]+\left[\begin{array}{c}
\mathbf{0}\\
-\mathbf{G}^{(N)}\left(\mathbf{r}(t)\right)
\end{array}\right],\label{eq:FirstOrderSystem}
\end{equation}
and the orbit determination problem is now given by (\ref{eq:intEqELP})
with appropriate forces as follow from (\ref{eq:FirstOrderSystem}).
Using (\ref{eq:intEqELP}) accelerates convergence of the fixed point
iteration in our scheme.

\subsection{Example}

We present an example of using our method. An extensive study of the
method for applications in astrodynamics may be found in \cite{B-J-B-S-A:2013}
(see also \cite{B-J-B-A:2012p}) and here we simply demonstrate that
our scheme allows computations on large time intervals and requires
relatively few evaluations of the full gravity model. Since the cost
of evaluating the full (high-degree) gravity model is substantial,
this results in significant computational savings.

As an example, we simulate an orbit with initial condition 
\[
\mathbf{r}_{\left|t=0\right.}=\left(\begin{array}{c}
x_{0}\\
y_{0}\\
z_{0}
\end{array}\right)=\left(\begin{array}{r}
2284.060\\
6275.400\\
4.431
\end{array}\right)\ ({\rm km)}
\]
 and 
\[
\frac{d\mathbf{r}}{dt}_{\left|t=0\right.}=\mathbf{v}_{0}=\left(\begin{array}{r}
-5.947\\
2.164\\
0
\end{array}\right)\,({\rm km/s}),
\]
and propagate it for 86,000 seconds (approximately 1 day). We use
$22$ time intervals and, on each interval, quadratures with $74$
nodes. Hence, on average, this corresponds to time distance between
nodes of approximately $53$ seconds. For the full gravitational model
we use a $70$ degree spherical harmonics model WGS84 \cite{WGS:1984}.

Using the 8th-order Gauss-Jackson integration scheme with very fine
sampling (one second time step), we generate the reference solution.
We selected the Gauss-Jackson method since it is often used for orbit
computations in astrodynamics; we refer to \cite{B-J-B-S-A:2013,B-J-B-A:2012p}
for a more detailed discussion on the issue of generating reference
solutions.

We then compute the orbit trajectory using the algorithm from Section~\ref{sec:alg}
adopted to the problem of orbit propagation as described in Section~\ref{sub:Problem-of-Orbit}
and compare the result with the reference solution. Achieving an error
of less than $5$ cm at the final time, we need $6512$ evaluations
of the reduced (3-term) gravitational model, and $3256$ evaluations
of the full gravitational model.

\section{Conclusions}

We have constructed an implicit, symplectic integrator that has speed
comparable to explicit multistep integrators currently used for orbit
computation. The key difference with the traditional IRK method is
that our scheme uses quadratures for band-limited exponentials rather
than the traditional Gaussian quadratures constructed for the orthogonal
Legendre polynomials. The nodes of quadratures for band-limited exponentials
do not concentrate excessively towards the end points of an interval
thus removing a practical limit on the number of nodes used within
each time interval.

\section{Appendix}

In both approaches described below we use nodes of generalized Gaussian
quadratures for exponentials $\left\{ \tau_{l}\right\} _{l=1}^{M}$
constructed in \cite{BEY-MON:2002} (see Lemma~\ref{lem:Quadratures}).
Some of the steps may require extended precision to yield accurate
results.

\subsection{Computing integration matrix using exact PSWFs}

In this approach we assume that the solutions $\psi_{j}^{c}(x)$ and
the eigenvalues $\lambda_{j}$ satisfying 
\begin{equation}
\left(F_{c}\psi_{j}^{c}\right)(x)=\int_{-1}^{1}e^{icxy}\psi_{j}^{c}(y)dy=\lambda_{j}\psi_{j}^{c}(x),\label{eq:EigProblem}
\end{equation}
where $F_{c}$ is defined in (\ref{PSWFdef}), are available. We use
(\ref{eq:ExactInterpolatingFnc}) and the matrix of values of PSWFs
at the nodes, $\psi_{j}^{c}(\tau_{l})$, to compute coefficients $\alpha_{kj}$,
so that we have 
\[
R_{k}^{c}(\tau)=\sum_{j=0}^{M-1}\alpha_{kj}\psi_{j}^{c}(\tau),\,\,\,\, k=1,\dots M.
\]
We then compute weights using (\ref{eq:DefWeights}),
\[
w_{k}=\int_{-1}^{1}R_{k}^{c}(x)dx=\sum_{j=0}^{M-1}\alpha_{kj}\int_{-1}^{1}\psi_{j}^{c}(x)dx=\sum_{j=0}^{M-1}\alpha_{kj}\lambda_{j}\psi_{j}^{c}(0).
\]
Next we define
\[
K_{l}^{c}(x)=\int_{-1}^{x}R_{l}^{c}(s)ds=\sum_{j=0}^{M-1}\alpha_{lj}\int_{-1}^{x}\psi_{j}^{c}(s)ds=\sum_{j=0}^{M-1}\alpha_{lj}\Phi_{j}^{c}(x),
\]
where 
\begin{equation}
\Phi_{j}^{c}(x)=\int_{-1}^{x}\psi_{j}^{c}(s)ds.\label{eq:Primitive}
\end{equation}
In order to compute the integration matrix (\ref{eq:IntegrationMatrixDef}),
we need to evaluate 
\[
w_{k}S_{kl}=\int_{-1}^{1}K_{l}^{c}(x)R_{k}^{c}(x)dx=\sum_{j,j'=0}^{M-1}\alpha_{lj}\alpha_{kj'}\int_{-1}^{1}\Phi_{j}^{c}(x)\psi_{j'}^{c}(x)dx=\sum_{j,j'=0}^{M-1}\alpha_{lj}\alpha_{kj'}I_{jj'},
\]
where 
\begin{equation}
I_{jj'}=\int_{-1}^{1}\Phi_{j}^{c}(x)\psi_{j'}^{c}(x)dx=\int_{-1}^{1}\Phi_{j}^{c}(x)\frac{d}{dx}\Phi_{j'}^{c}(x)dx.\label{eq:IntegralToCompute}
\end{equation}
We have
\begin{prop}
If $j$ and $j'$ are both even, then
\begin{equation}
I_{jj'}=I_{j'j}=\frac{1}{2}\lambda_{j}\lambda_{j'}\psi_{j}^{c}(0)\psi_{j'}^{c}(0).\label{eq:even-even}
\end{equation}
If $j$ and $j'$ are both odd, then
\begin{equation}
I_{jj'}=0.\label{eq:odd-odd}
\end{equation}
If $j$ is even and $j'$ is odd, then
\begin{equation}
I_{jj'}=-I_{j'j},\label{eq:odd-even-0}
\end{equation}
\begin{equation}
I_{jj'}=\frac{\lambda_{j'}}{ic\lambda_{j}}\int_{-1}^{1}\psi_{j}^{c}(y)\frac{\psi_{j'}^{c}(y)}{y}dy\label{eq:odd-even-1}
\end{equation}
and
\begin{equation}
I_{j'j}=\frac{\lambda_{j}}{ic\lambda_{j'}}\left(\int_{-1}^{1}\psi_{j}^{c}(y)\frac{\psi_{j'}^{c}(y)}{y}dy-2\psi_{j}^{c}(0)\int_{0}^{1}\frac{\psi_{j'}^{c}(y)}{y}dy+ic\psi_{j}^{c}(0)\overline{\lambda}_{j'}\int_{0}^{1}\psi_{j'}^{c}(y)dy\right).\label{eq:odd-even-2}
\end{equation}
 
\end{prop}
We use (\ref{eq:odd-even-1}) if $\left|\lambda_{j'}\right|<\left|\lambda_{j}\right|$,
(\ref{eq:odd-even-2}) otherwise. 
\begin{proof}
Integrating (\ref{eq:IntegralToCompute}) by parts, we obtain
\begin{equation}
I_{jj'}+I_{j'j}=\Phi_{j}^{c}(1)\Phi_{j'}^{c}(1)-\Phi_{j}^{c}(-1)\Phi_{j'}^{c}(-1)=\lambda_{j}\lambda_{j'}\psi_{j}^{c}(0)\psi_{j'}^{c}(0)\label{eq:relation}
\end{equation}
and, since $\psi_{j}(0)=0$ if $j$ is odd (due to parity of PSWFs),
we arrive at (\ref{eq:odd-odd}) and (\ref{eq:odd-even-0}).

Using (\ref{eq:Primitive}) and (\ref{eq:EigProblem}), we have 
\[
\Phi_{j}^{c}(x)=\frac{1}{\lambda_{j}}\int_{-1}^{1}\left(\int_{-1}^{x}e^{icys}ds\right)\psi_{j}^{c}(y)dy=\frac{1}{\lambda_{j}}\int_{-1}^{1}\frac{e^{icyx}-e^{-icy}}{icy}\psi_{j}^{c}(y)dy,
\]
and, thus,
\begin{eqnarray}
I_{jj'} & = & \frac{1}{\lambda_{j}}\int_{-1}^{1}\left[\int_{-1}^{1}\frac{e^{icyx}-e^{-icy}}{icy}\psi_{j}^{c}(y)dy\right]\psi_{j'}^{c}(x)dx\nonumber \\
 & = & \frac{\lambda_{j'}}{ic\lambda_{j}}\left(\int_{-1}^{1}\psi_{j}^{c}(y)\frac{\psi_{j'}^{c}(y)}{y}dy-\psi_{j'}^{c}(0)\int_{-1}^{1}\frac{\psi_{j}^{c}(y)}{y}e^{-icy}dy\right).\label{eq:integral-for-derivation}
\end{eqnarray}
It follows from (\ref{eq:integral-for-derivation}) that if $j$ is
even and $j'$ is odd (so that $\psi_{j'}^{c}(0)=0$), we obtain (\ref{eq:odd-even-1})
and 
\[
I_{j'j}=\frac{\lambda_{j}}{ic\lambda_{j'}}\left(\int_{-1}^{1}\psi_{j}^{c}(y)\frac{\psi_{j'}^{c}(y)}{y}dy-\psi_{j}^{c}(0)\int_{-1}^{1}\frac{\psi_{j'}^{c}(y)}{y}e^{-icy}dy\right).
\]
Introducing
\[
u(x)=\int_{-1}^{1}\frac{\psi_{j'}^{c}(y)}{y}e^{-icyx}dy,
\]
we have
\[
u'(x)=-ic\int_{-1}^{1}\psi_{j'}^{c}(y)\, e^{-icyx}dy=-ic\overline{\lambda}_{j'}\psi_{j'}^{c}(x)
\]
so that 
\[
u(x)=u(a)-ic\overline{\lambda}_{j'}\int_{a}^{x}\psi_{j'}^{c}(s)ds.
\]
Setting $x=1$ and $a=0$, we obtain
\begin{eqnarray*}
\int_{-1}^{1}\frac{\psi_{j'}^{c}(y)}{y}e^{-icy}dy & = & \int_{-1}^{1}\frac{\psi_{j'}^{c}(y)}{y}dy-ic\overline{\lambda}_{j'}\int_{0}^{1}\psi_{j'}^{c}(y)dy\\
 & = & 2\int_{0}^{1}\frac{\psi_{j'}^{c}(y)}{y}dy-ic\overline{\lambda}_{j'}\int_{0}^{1}\psi_{j'}^{c}(y)dy
\end{eqnarray*}
and arrive at (\ref{eq:odd-even-2}).
\end{proof}

\subsection{Computing integration matrix using approximate PSWFs}

If the interpolating basis for band-limited functions is defined via
(\ref{eq:rkx}), then the coefficients $r_{kl}$ are obtained using
\begin{equation}
\delta_{km}=R_{k}(\tau_{m})=\sum_{l=1}^{M}r_{kl}e^{ic\tau_{l}\tau_{m}}\label{eq:IntProl}
\end{equation}
by inverting the matrix $E=\left\{ e^{ic\tau_{l}\tau_{m}}\right\} _{l,m=1,\dots M}$.
We have

\[
K_{k}(x)=\int_{-1}^{x}R_{k}(s)ds=\sum_{l=1}^{M}r_{kl}\frac{e^{ic\tau_{l}x}-e^{-ic\tau_{l}}}{ic\tau_{l}}
\]
and compute 
\begin{eqnarray*}
w_{k}S_{kl} & = & \int_{-1}^{1}K_{l}(x)R_{k}(x)dx\\
 & = & \sum_{j,j'=1,\dots M}r_{kj}r_{lj'}\int_{-1}^{1}e^{ic\tau_{j}x}\frac{e^{ic\tau_{j'}x}-e^{-ic\tau_{j'}}}{ic\tau_{j'}}dx\\
 & = & \sum_{j,j'=1,\dots M}r_{kj}r_{lj'}G_{jj'},
\end{eqnarray*}
where
\[
G_{jj'}=2\frac{\mbox{sinc}\left(c\left(\tau_{j}+\tau_{j'}\right)\right)-e^{-ic\tau_{j'}}\mbox{sinc}\left(c\tau_{j}\right)}{ic\tau_{j'}}.
\]
 Thus, we have
\[
w_{k}S_{kl}=\left(E^{-1}GE^{-1}\right)_{kl}.
\]

\bibliographystyle{plain}
%\bibliography{common}

\end{document}